\documentclass[fontsize = 12]{article}

\usepackage{amsthm,amsmath,amsfonts,amssymb}
\usepackage[colorlinks, citecolor=blue, urlcolor=blue]{hyperref}
\usepackage[authoryear]{natbib}
\usepackage{graphicx}
\usepackage{color, soul}
\usepackage{setspace}
\theoremstyle{plain}
\newtheorem{theorem}{Theorem}[section]
\newtheorem{corollary}{Corollary}[section]

\title{Location Characteristics of Conditional Selective Confidence Intervals via Polyhedral Methods\thanks{ The authors would like to thank participants at Kansai Econometric Study Group for valuable comments. All errors are our own. Dzemski acknowledges financial support from Jan Wallanders och Tom Hedelius stiftelse samt Tore Browaldhs stiftelse under grant P24-0135. Okui acknowledges financial support from the Japan Society for the Promotion of Science under KAKENHI Grant Nos.23K25501. Wang acknowledges financial support from the Singapore Ministry of Education Tier 1 grants RG104/21, RG51/24, and NTU CoHASS Research Support Grant.}}

\author{Andreas Dzemski\thanks{Department of Economics, University of Gothenburg, P.O. Box 640, SE-405 30, Gothenburg, Sweden. Email: andreas.dzemski@economics.gu.se}, Ryo Okui\thanks{Faculty of Economics, the University of Tokyo, Bunkyo-ku, Tokyo 113-0033, Japan. Email: okuiryo@e.u-tokyo.ac.jp}, and Wenjie Wang\thanks{Division of Economics, School of Social Sciences, Nanyang Technological University. HSS-04-65, 14 Nanyang Drive, Singapore 637332. Email: wang.wj@ntu.edu.sg.}}

\date{\today}
\begin{document}
\maketitle

\begin{abstract}
We examine the location properties of a conditional selective confidence interval constructed via the polyhedral method. The interval is derived from the distribution of a test statistic conditional on the event of statistical significance. For a one-sided test, its behavior depends on whether the parameter is highly or only marginally significant. In the highly significant case, the interval closely resembles the conventional confidence interval that ignores selection. By contrast, when the parameter is only marginally significant, the interval may shift far to the left of zero, potentially excluding all \emph{a priori} plausible parameter values. This ``location problem'' does not arise if significance is determined by a two-sided test or by a one-sided test with randomized response (e.g., data carving).

\par\vskip\baselineskip\noindent
Multiple hypothesis testing, post-selection inference, conditional inference.
\end{abstract}

Lately, selective inference has received increased attention from researchers. A partial list of papers includes \cite{berk_valid_2013,lockhart_significance_2014,fithian_optimal_2017,tibshirani_exact_2016, lee2016exact, bachoc_valid_2019, watanabe2021, jewell_testing_2022,terada2023,duy2023}, and a review is available in \cite{zhang2022post}.

Methods for selective inference provide valid statistical inference on parameters when the estimated model or the parameters of interest are selected based on the observed data. Conventional inference methods assume a fixed model structure and require pre-specification of the parameters of interest before analysis. Empirical practice often departs from this ideal, potentially invalidating conventional methods. For example, standard confidence intervals for parameters that are chosen based on the observed data may not have the desired coverage. Data-dependence may be introduced by formal methods for model selection such as the Lasso \citep{lee2016exact}, or from more informal practices of data-mining. An example of informal data-mining is a $p$-hacking scenario where a researcher estimates many parameters and then crafts a narrative highlighting the few that are statistically significant. 

While various methods have been proposed for selective inference, there has been limited research on the specific properties of the results produced by these methods with notable exceptions by \cite{kivaranovic2021,kivaranovic2024}. This note addresses such a gap in the literature by examining the characteristics of conditional selective confidence intervals constructed using the polyhedral approach described by \cite{lee2016exact}. We are particularly interested in the behavior of the confidence intervals at the margin of the conditioning set. Our main finding is that the conditional selective confidence interval may be positioned in extreme ranges and can deviate greatly from the conventional parameter estimate. 

Our setting is the ``file drawer problem'' \citep{rosenthal1979}, that is, the situation where only statistically significant estimates are presented.
The main result is on inference on a parameter $\theta$ that is found significant by a one-sided $t$-test. Let $X$ denote the corresponding $t$-statistic and $ c $ the critical value. Suppose that we observe $X = x^{\text{obs}}$. The parameter is significant if $x^{\text{obs}} \geq  c $.
The polyhedral method for selective inference on $\theta$ is based on the distribution of $X$ conditional on observing $X \geq  c $, which has the distribution function $F(a; \theta,  c ) := \Pr (X \leq a \mid X \geq  c )$. Note that $F$ depends on $\theta$ because of the dependence of the distribution of $X$ on it. The conditional confidence interval is found by evaluating the distribution function at the observed $t$-statistic $x^{\text{obs}}$ and then inverting $\theta \mapsto F(x^{obs}; \theta,  c )$.

Our main result is for situations where $x^{\text{obs}}$ only barely exceeds $ c $, a case often referred to as ``marginal significance''.
For this case, we show that the conditional confidence interval is very wide and located far away from any reasonable parameter values. Specifically, we show that even the upper bound of the interval diverges to $-\infty$ as $x^{\text{obs}} \downarrow  c $. Therefore, at the margin of significance, the confidence interval contains only highly negative values that may be a-priori implausible.

Our results are closely related to \cite{kivaranovic2021}, who show that the expected length of a confidence interval constructed by the polyhedral method is infinite when the conditioning set is bounded either from below or from above. Indeed, our main result immediately follows from their Lemma A.3. However, our focus differs from theirs. They consider the expected length of the confidence interval. We complement their results by identifying precisely when (that is, in what kind of samples) the confidence interval exhibits extreme behavior.
Whereas \cite{kivaranovic2021} focus on the length of the confidence interval, our focus is on its location.
We show that the conditional confidence interval for a marginally significant parameter can be located in an extreme range of the real line, far to the left of zero. In particular, a researcher who expects the parameter to lie within a certain ``plausible range'' may find that all plausible values to be excluded from the conditional confidence interval. 

This ``location problem'' applies more broadly than our file drawer problem, since it builds on Lemma A.3 of \cite{kivaranovic2021} which holds generally for truncated normal distributions. In particular, when the conditioning set is bounded from either below or above and the observed parameter is at the margin of the conditioning set, the conditional selective confidence intervals are located in extreme ranges.

Our arguments are entirely mechanical and do not rely on the underlying distribution of the data. Regardless of the data's true distribution or the parameter's true value, if the observed value of the $t$-statistic is close to the critical value, the problem identified in this paper will arise. This aspect of our results starkly contrasts other studies that highlight potential issues with selective inference methods. For instance, the result in \cite{kivaranovic2021} on the expected length relies on  distributional assumptions to evaluate the expectation. Our result does not need any distributional assumptions. 

As far as we know, the location problem has not been previously discussed in the literature. 
The prevalence of ``marginally significant'' results in empirical work \citep{gerber2008statistical,brodeur2016star,brodeur2020methods} suggests that it has great practical importance.
Using conditional inference to account for the selectivity of significant results leaves researchers vulnerable to the location problem, potentially preventing them from inferring plausible parameter values. As we discuss below, researchers can avoid this issue by a careful choice of selection mechanism.

Our analysis of polyhedral inference considers also the cases of a moderately or highly significant estimate in addition to the case of marginal significance.

An estimate is ``highly significant'' if its $t$-statistic far exceeds the critical value. In such cases, the conditional selective confidence interval is nearly indistinguishable from the conventional interval. This property has already been noted in the literature. For example, \cite{andrews_inference_2024} study inference for identifying the most welfare-improving policy and show that, when the best policy is unambiguously determined, the confidence interval based on the polyhedral method is almost identical to the conventional one.

For moderate cases, where the parameter estimate is neither ``marginally'' nor ``highly'' significant, the conditional selective confidence interval reveals non-trivial information about the population parameter. It is wider than the conventional interval, but still informative enough to provide a useful range of parameter values. The increased length intuitively reflects additional statistical uncertainty when the data is used both for selection and inference. 

We discuss two modifications to the selection rule in the file drawer problem that eliminate the location problem. The suitability of these modifications depends on the empirical context.

Firstly, we consider selection with a randomized response \citep{fithian_optimal_2017, tian_selective_2018, panigrahi21integrative, panigrahi_approximate_2023}. Conditional inference under this selection rule has been shown to yield confidence intervals with bounded expected lengths \citep{kivaranovic2024}. We find that such methods avoid the location problem.

Secondly, we replace the one-sided $t$-test by a two-sided $t$-test. With this modification, polyhedral inference does not exhibit the location problem. This is already suggested by the result of \cite{kivaranovic2021} that the expected length of the interval is finite in this case.

The remainder of the paper is organized as follows. Section \ref{sec: setting} introduces the settings and illustrates the problems with a numerical example. Section \ref{sec: theory} presents the main analytical results. We then discuss alternative scenarios in which the problems can be avoided. Specifically, inferences with randomized response are addressed in Section \ref{sec:randomized response}, and the two-sided significance cases are explored in Section \ref{sec: two-sided}. Finally, Section \ref{sec: conclusion} concludes the discussion.

\section{Setting} 
\label{sec: setting}

Our setting is the file drawer problem from \cite{rosenthal1979} of reporting only statistically significant effects.  We consider conditional inference based on the polyhedral approach form \cite{lee2016exact}.

We are interested in a parameter $\theta$ for which an estimator $\hat \theta$ is available. Assume that the estimator $\hat \theta$ is unbiased and Gaussian: $\hat \theta \sim N (\theta, \sigma^2)$. For simplicity, we assume that the variance $\sigma^2$ is known. In this case, it is without loss of generality to normalize $\sigma^2 = 1$ and $\hat{\theta}$ is identical to its $t$-statistic $X = X/\sigma$. Let $x^{\text{obs}}$ denote the observed value of $X$.
 First, we discuss the one-sided significance, where an effect is significant if $x^{\text{obs}} \geq c$ for a pre-specified critical value $ c $. The two-sided significance case is discussed in Section \ref{sec: two-sided}.

The polyhedral approach is based on the conditional distribution of the $t$-statistic $X$ given that it exceeds the critical value. 
Under Gaussianity the corresponding conditional distribution function is given by
\begin{align}
    \label{eq: conditional distribution file drawer}
    F(a; \theta,  c ) :=& \Pr (X \leq a \mid X \geq  c )
    \\
    \notag
    =& \frac{\Phi (a - \theta) - \Phi ( c  - \theta)}{1- \Phi(  c  - \theta)}.
\end{align}
Let $\theta(p)$ denote the value of $\theta$ that solves the equation
\begin{align*}
    p =  F( x^{\text{obs}}; \theta,  c )
\end{align*}
for given $p \in (0, 1)$, $x^{\text{obs}}$, and $ c $. The solution exists by Lemma~A.2 in \cite{kivaranovic2021} and is strictly decreasing in $p$. The dependence of $\theta(p)$ on $x^{\text{obs}}$ is crucial for the results in this paper, even though it is kept implicit in the notation.

Based on $\theta(p)$ we now define a conditionally media-unbiased estimator and a conditional confidence set.
The conditionally median-unbiased estimator is given by
\begin{align*}
    \hat \theta_{MU} := \theta (0.5).
\end{align*}
A $100(1-\alpha)$\% conditional confidence interval is defined as 
\begin{align*}
    [\theta (1- \alpha/2), \theta(\alpha/2)].
\end{align*}
For example, $[\theta (0.975), \theta (0.025)]$ is the 95\% confidence interval.

The validity of the confidence interval is verified by the following:
\begin{align*}
    &\Pr ( \theta \in  [\theta (1- \alpha/2), \theta(\alpha/2)] \mid X \geq  c ) \\
    =& \Pr ( \alpha/2 \leq F( X, \theta,  c )  \leq 1 - \alpha/2 \mid X \geq  c )
    = 1- \alpha,
\end{align*}
which relies on the fact that, conditional on $X \geq  c $, $F (X, \theta,  c )$ is uniformly distributed on the unit interval.

\label{sec: numerical illustration}

To provide numerical evidence of the key issue discussed in this note, we set $c = 1.64$, corresponding to a nominal level of 5\% for one-sided significance testing. We set the targeted conditional coverage of the confidence interval to 95\%.

\begin{figure*}
    \centering
    \caption{Test statistic and confidence intervals}
    \label{fig: numerical illustration}
    \includegraphics[width = 0.8\linewidth]{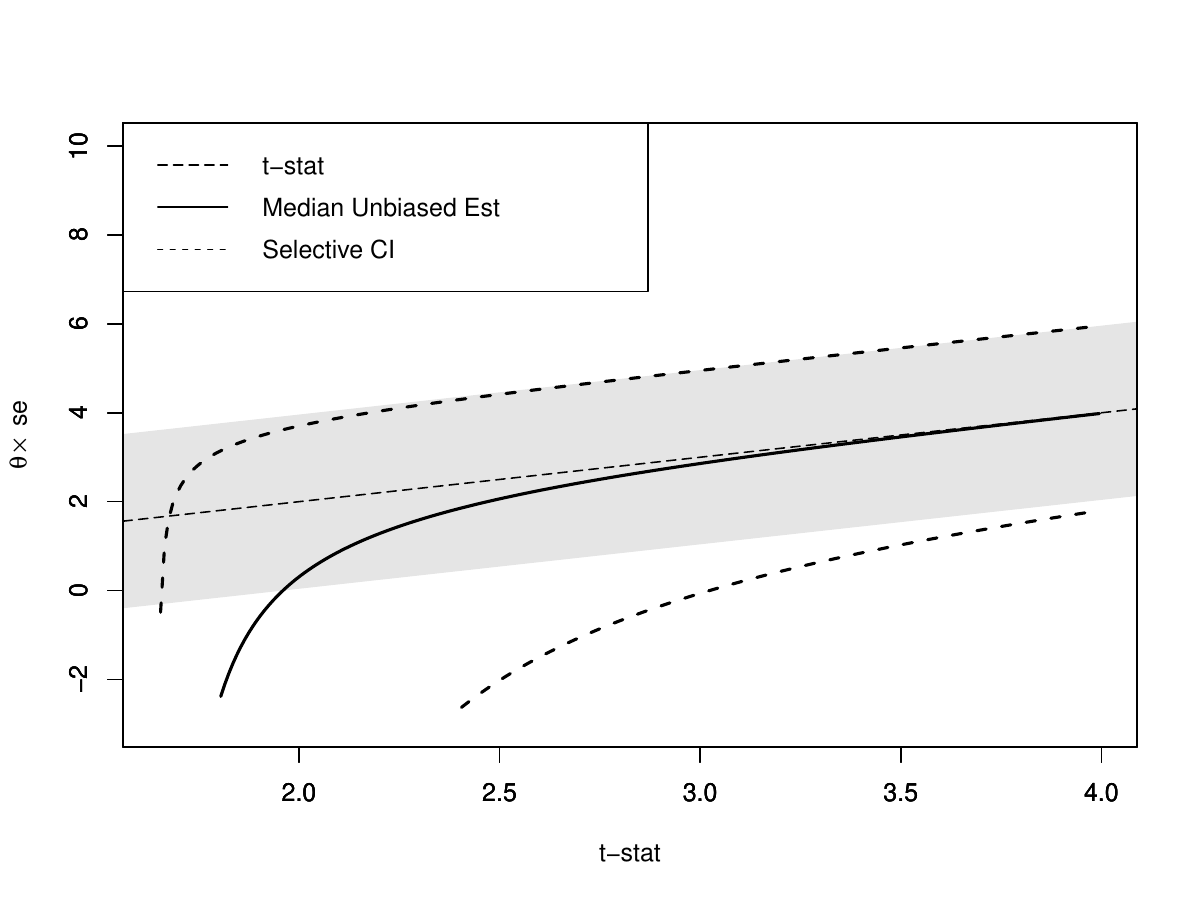}
    \begin{minipage}{0.8\linewidth}
    \footnotesize      NOTE: Conditional inference after one-sided testing with critical value $ c  = 1.64$. The figure plots the median-unbiased estimator (solid curve) and the bounds $\theta(0.975)$ and $\theta(0.025)$ of the 95\% conditional confidence interval (dashed curves). The thin dashed line is the 45-degree line, representing the unconditional estimate. The shaded region represents the conventional unconditional 95\% confidence interval.
    \end{minipage}
\end{figure*}

Figure~\ref{fig: numerical illustration} plots the median-unbiased estimator and the conditional confidence interval as a function of the observed $t$-statistic $x^{\text{obs}}$. The upper dashed curve represents the upper bound $\theta(0.025)$ of the confidence interval and the lower dashed curve represents the lower bound $\theta(0.975)$. The solid curve represents the median unbiased estimator $\hat \theta_{MU}$. The thin dashed line is the 45-degree line. Since we have normalized $X = \hat \theta$, it gives the unconditional estimate $\hat{\theta}$. The shaded region represents the conventional confidence interval $x^{\text{obs}} \pm 1.96$. 

Figure \ref{fig: numerical illustration} reveals several noteworthy insights. When the estimated parameter is highly significant, that is, when $x^{\text{obs}}$ is large, selective inference and conventional inference yield almost identical results. In this case, the median-unbiased estimator is close to the conventional parameter estimate ($\hat{\theta}_{MU} \approx \hat{\theta}$) and the conditional confidence interval approximates the traditional equal-tailed interval, that is, $\theta(0.975) \approx x^{\text{obs}} - 1.96$ and $\theta(0.025) \approx x^{\text{obs}} + 1.96$.

When the observed $t$-statistic is of moderate size, conditional inference results differ from the conventional inference. The median-unbiased estimator lies below the conventional estimate, and the conditional confidence interval is wider than its conventional counterpart.
Its lower bound falls below the lower bound of the standard interval, whereas the upper bounds of the two intervals remain fairly similar. These downward adjustments reflect a correction for selection bias: significant effects are positively selected because overestimated effects are more likely to exceed the threshold of significance. The wider interval reflects increased statistical uncertainty when the same data are used for both parameter selection and inference.

For ``marginally significant'' results, the downward correction can be quite extreme. As $x^{\text{obs}}$ approaches $ c $ from above, the upper bound of the conditional confidence interval falls below the critical value $ c $ and eventually below zero. 
Inferring negative values for a ``significantly positive'' parameter is not the contradiction it may appear: the significance test controls the \emph{unconditional} error, whereas the conditional confidence interval gives a conditional guarantee. From a practical perspective, however, a confidence interval that includes large negative values may be deemed implausible, and an empirical researcher may reasonably conclude that the inference results are uninformative.

The numerical results in Figure~\ref{fig: numerical illustration} are easily analytically verified.
The solution $\theta(p)$ depends on the observed $t$-statistic $x^{\text{obs}}$. To find the range of $x^{\text{obs}}$ such that $\theta(p) <  c $, solve the inequality 
\begin{align*}
    p > F(x^{\text{obs}};  c ,  c ) = \frac{\Phi(x^{\text{obs}}- c ) - 0.5}{0.5},
\end{align*}
for $x^{\text{obs}}$ to obtain
\begin{align*}
    x^{\text{obs}} < \Phi^{-1} ( 0.5p + 0.5) +  c .
\end{align*}
Thus, unless $x^{\text{obs}}$ exceeds $ c $ by more than $\Phi^{-1} (0.5p+0.5)$, $\theta(p)$ is smaller than $ c $. 
For example, 
if $x^{obs} < \Phi^{-1} (0.5 \times 0.025+0.5) +  c  \approx 0.031+   c $, the upper bound of the 95\% confidence interval is less than $ c $.

Similarly, we can show that we have $\theta(p) < 0$ for 
\begin{align*}
    x^{\text{obs}} < \Phi^{-1} ( (1- \Phi ( c ) ) p + \Phi ( c ) ) .
\end{align*}
Thus, unless $x^{\text{obs}}$ exceeds $\Phi^{-1} ( (1- \Phi ( c ) ) p + \Phi ( c ) ) $, $\theta(p)$ is negative.
In particular, the upper bound of a 95\% confidence interval is negative when $x^{\text{obs}} < \Phi^{-1} (0.05 \times 0.025+0.95) \approx 1.66 $. This means that the estimate $\hat{\theta}$ is significantly positive at the selection stage, but significantly negative at the inference stage. This apparent contradiction can be reconciled by noting that the selection stage is based on the unconditional distribution of $\hat{\theta}$, whereas the inference stage is based on its conditional distribution.

\section{Theoretical results}
\label{sec: theory}

In this section, we analytically derive the behavior of conditional selective inference under the ``marginally significant'' and ``highly significant'' cases. For the former, we consider the behavior of $\theta(p)$ when $x^{\text{obs}}$ is close to the critical value $ c $. For the latter, we consider the behavior of $\theta(p)$ for large values of the observed $t$-statistic $x^{\text{obs}}$. 

For the case of marginal significance, Lemma~A.3 in \cite{kivaranovic2021} implies that, as $x^{\text{obs}}$ approaches $c$, the conditional confidence interval and the median-unbiased estimator drift toward negative infinity.

\begin{theorem}
    \label{thm: marginally significant}
  For $p \in (0, 1)$, as $x^{\text{obs}} \downarrow  c $
  \begin{align*}
    \theta (p) \to - \infty.
  \end{align*}
\end{theorem}

\begin{proof}
See Lemma~A.3 in \cite{kivaranovic2021}.
\end{proof}
This result holds regardless of the true distribution of $X$. Its proof is entirely based on the functional form of $F(a, \theta,  c )$ and does not rely on the probabilistic characteristics of the sample. 

Theorem~\ref{thm: marginally significant} implies that, as the estimated effect size approaches insignificance, every possible parameter value is ruled out eventually. This is formalized in the following corollary:
\begin{corollary}
    For any $\theta \in \mathbb{R}$ and $\alpha \in (0,0.5)$, there exists $\delta >0$ such that if $c < x^{\text{obs}} < c+ \delta$, then $\theta \notin [ \theta (1-\alpha/2), \theta(\alpha/2)]$.
\end{corollary}
For a scenario where one has information or beliefs about a plausible range of parameter values, this means that, if $x^{\text{obs}}$ realizes too close to the margin of significance, the conditional confidence interval will not contain any plausible values. 

Turning to the case of a highly significant parameter estimate, note that the conventional equal-tailed confidence interval at the $(1-\alpha)$-level is given by the interval $[\theta^*(1 - \alpha/2), \theta^*(\alpha/2)]$ with $\theta^*(p) = x^{\text{obs}} - \Phi^{-1}(p)$. For example, the conventional 95\% confidence interval is $[x^{\text{obs}} - 1.96, x^{\text{obs}} + 1.96]$. 
Similarly to $\theta(p)$, $\theta^*(p)$ can also be found by inverting a distribution function. In particular, $\theta^*(p)$ is the value of $\theta$ that solves the equation $p = \Phi (x^{\text{obs}} - \theta)$.
The conventional point estimate is given by $\hat{\theta} = \theta^*(0.5)$.

The following result establishes that, as $x^{\text{obs}} \uparrow \infty$, the conditional selective confidence interval converges to the conventional confidence interval and the conditional median-unbiased estimate converges to the conventional estimate.

\begin{theorem}
  For $p \in (0, 1)$, as $x^{\text{obs}} \uparrow \infty$,
  \begin{align*}
    \theta (p) - \theta^*(p) \to 0.
  \end{align*}
\end{theorem}

\begin{proof}
    Fix $p \in (0,1)$ and take $\xi >0$ arbitrarily. To show that $\lvert \theta(p) - \theta^*(p) \rvert \leq \xi$ eventually, let $\delta > 0$ be small enough that 
\begin{align*}
    \Phi^{-1} (p + \delta) - \Phi^{-1} (p - \delta) \leq \xi.
\end{align*}
We can find such $\delta$ by the continuity of $\Phi^{-1}$. 
Since $\theta^*(a) = x^{\text{obs}} - \Phi^{-1}(a)$, it suffices to show that, for $x^{\text{obs}}$ large enough, 
\begin{align*}
    \theta^*(p + \delta)
    \leq \theta(p) \leq
    \theta^*(p - \delta).
\end{align*}
Because $ c $ is fixed, for every $\epsilon > 0$, there exists $M(\epsilon)$ such that for any $x^{\text{obs}} > M(\epsilon)$, 
\begin{align}
    \label{eq: epsilon}
    \Phi ( c  - (x^{\text{obs}} - \Phi^{-1} (p))) < \epsilon.
\end{align}
Let $\epsilon > 0$ be small enough that
\begin{align*}
    (p - \delta) / (1 - \epsilon) < p < p + \delta - \epsilon.
\end{align*}
By inequality~\eqref{eq: epsilon}, for $x^{\text{obs}} > M(\epsilon)$,
\begin{align*}
    F(x^{\text{obs}} ;  \theta^*(p - \delta),  c ) \leq & \frac{ \Phi (x^{\text{obs}} - \theta^*(p - \delta) ) }{1- \epsilon}
    \\
    \leq & \frac{p - \delta}{1- \epsilon} < p.
\end{align*}
Since $F(x^{\text{obs}} ;  \theta(p),  c ) = p$ and $F(x^{\text{obs}} ;  \theta,  c )$ is strictly decreasing in $\theta$ by Lemma~A.2 in \cite{kivaranovic2021}, this establishes $\theta(p) \leq \theta^*(p - \delta)$ for $x^{\text{obs}} > M(\epsilon)$. Also, by inequality~\eqref{eq: epsilon}, for $x^{\text{obs}} > M(\epsilon)$,
\begin{align*}
    F(x^{\text{obs}} ;  \theta^*(p + \delta),  c )
    \geq & \Phi (x^{\text{obs}} - \theta^*(p + \delta) ) - \epsilon
\\
=& p + \delta - \epsilon > p
\end{align*}
and therefore $\theta(p) \geq \theta^*(p + \delta)$ for $x^{\text{obs}} > M(\epsilon)$, concluding the proof. 
\end{proof}
\section{Alternative scenarios}

We now discuss two variations of the file drawer problem: file drawer with randomized response and file drawer with two-sided testing. In these alternative settings, it is possible to construct conditional inference procedures that avoid the location problem.

\subsection{\label{sec:randomized response}The randomized file drawer problem}

We first consider a modification of the file drawer problem that redefines significance to depend on an additional randomization. Selective inference in this setting falls within the framework of ``selection with a randomized response'' that has garnered attention in the recent statistics literature \citep{fithian_optimal_2017, tian_selective_2018,panigrahi21integrative, panigrahi_approximate_2023}. The randomization is expected to ``smooth out'' erratic behavior close to the threshold of significance. 

To describe the randomized file drawer problem, let $X\sim N(\theta, 1)$ as before and introduce a randomization $W$ that is independent of $X$ and satisfies $W \sim N(0,\eta^2)$. The parameter estimate $\hat{\theta} = X$ is ``significant after randomization'' if $X + W \geq c$ for a critical value $c$. The parameter $\eta$ is a tuning parameter that controls the degree of randomization. 

An example of a randomization mechanism is data carving \citep{fithian_optimal_2017}, that is, using a random subset of the data for selection and the entire sample for post-selection inference. In particular, suppose that we have a random sample $\{Z_i\}_{i=1}^n$, where $Z_i \sim N( \theta, n)$. The first $n_1$ observations determine the selection. Let
\begin{align*}
    X= \frac{1}{n} \sum_{i=1}^n Z_i, \quad  W= \frac{1}{n_1} \sum_{i=1}^{n_1} Z_i - \frac{1}{n} \sum_{i=1}^n Z_i.
\end{align*}
In this case, $X+W = \sum_{i=1}^{n_1} Z_i /n_1$. Note that $ X\sim N(\theta,1)$ and $ X+W \sim N(\theta, n/n_1)$ so that $\eta^2 = n/n_1 - 1 $.
To clarify the role of $\eta$, note that setting $n_1 = n/2$ (using half the sample for selection) gives $\eta^2 = 1$, while setting $n_1 = n$ (using the entire sample for both selection and inference) gives $\eta^2 = 0$. Selection and inference are closely related when $\eta$ is small, and they become more independent as $\eta$ increases.

To conduct statistical inference on the parameter of interest $\theta$, consider the conditional probability
\begin{align*}
    F^R (a; \theta, c) =  P( X \leq a \mid X+W \geq c).
\end{align*}
The only difference to \eqref{eq: conditional distribution file drawer} is that the conditioning set now accounts for the randomization.
Solve
\begin{align*}
    p =  F^R (x^{\text{obs}}; \theta^R (p), c),
\end{align*}
for $\theta^R (p)$.
The conditionally median-unbiased estimator is $\theta^R (0.5)$, and a $100(1-\alpha)\%$ conditional confidence interval is $[\theta^R (1- \alpha/2), \theta^R (\alpha/2)]$.\footnote{\cite{panigrahi_approximate_2023} consider the maximum likelihood estimation based on the conditional likelihood. In particular, Section 2 of their paper considers the file drawer problem.}

\begin{figure*}
    \centering
    \caption{Conditional inferences based on selection with a randomized response}
    \label{fig: randomized response}
    \includegraphics[width = 0.8\linewidth]{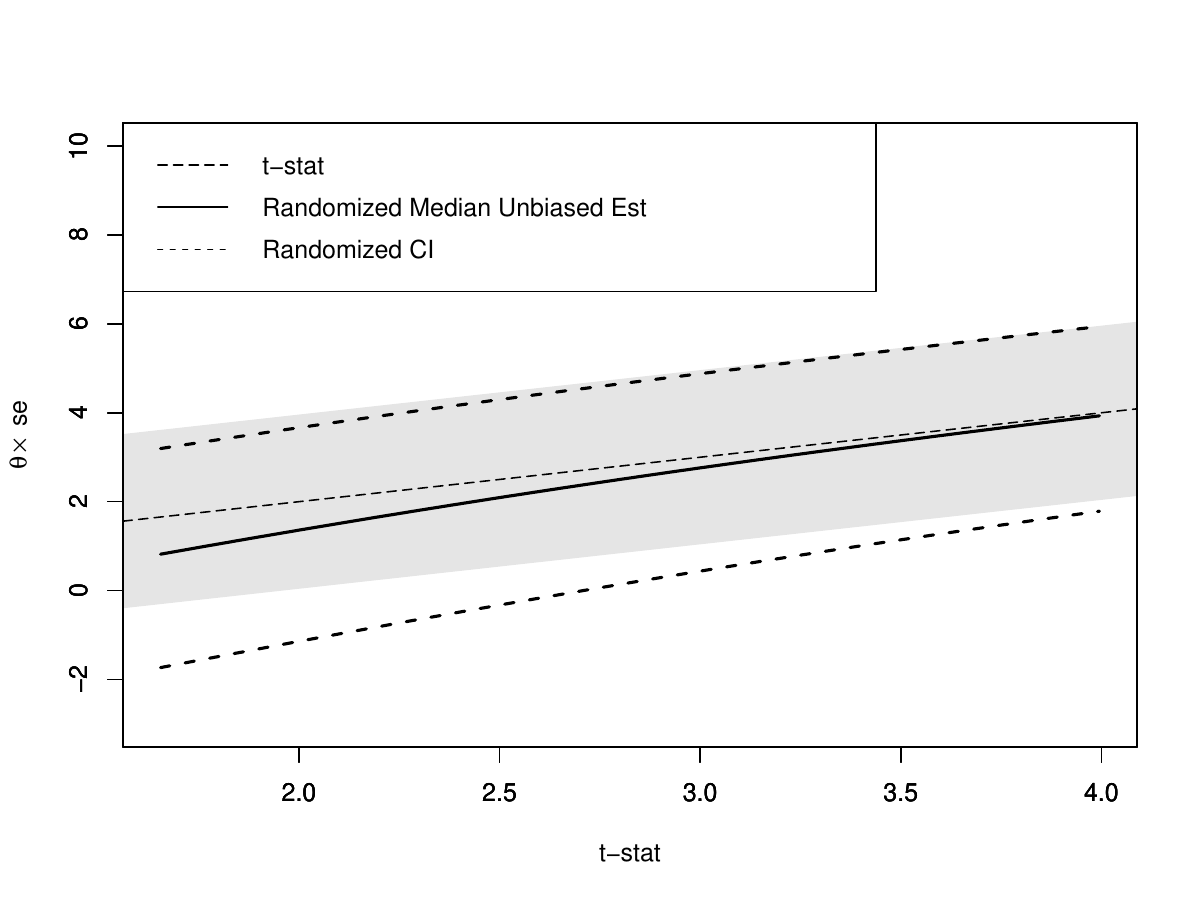}
    \begin{minipage}{0.8\linewidth}
    \footnotesize      NOTE: Conditional inference based on one-sided testing with a randomized response with $\eta^2 = 1$ and $c  = 1.64$. The figure plots the median-unbiased estimator (solid curve) and the bounds $\theta(0.975)$ and $\theta(0.025)$ of the 95\% conditional confidence interval (dashed curves). The thin dashed line is the 45-degree line, representing the unconditional estimate. The shaded region represents the conventional unconditional 95\% confidence interval.
    \end{minipage}
\end{figure*}

Figure~\ref{fig: randomized response} shows that selective inference with a randomized response avoids the location problem observed in the original file-drawer problem: both the median-unbiased estimator and the conditional confidence interval remain ``stable,'' in the sense that they do not diverge rapidly toward negative infinity, as $x^{\text{obs}}$ approaches the margin of significance.

The use of the randomized-response approach entails restrictions on the significance test. For example, in the case of data-carving, it requires that a substantial proportion of the data is set aside when testing significance. This corresponds to choosing a sufficiently large tuning parameter $\eta$. If $\eta$ is small, the inference results become similar to those of the original polyhedral method, potentially reintroducing the location problem. 
Testing significance with only part of the sample is suboptimal for empirical studies that focus on establishing significance and consider effect size only as a secondary concern. It is also not suitable for reexamining effect size in existing empirical studies that have already obtained significance results based on the entire sample.

Randomized response methods shine in empirical designs where effect size is the primary focus and where the selection procedure can be designed appropriately from the outset. For example, in high-dimensional settings where selection is employed primarily for dimension reduction --- using methods such as the Lasso --- the selection stage is not itself the central focus. In these cases, randomized-response techniques, including randomized Lasso \citep{panigrahi_approximate_2023}, can be especially advantageous.

\subsection{The two-sided file drawer problem}
\label{sec: two-sided}

In our original file drawer problem, the estimate is reported when it is ``significantly positive''. We now consider the alternative where estimates are reported if they are ``significantly different from zero'' (see Example~1 in \cite{fithian_optimal_2017}). Formally, we select estimates that satisfy $|x^{\text{obs}}| \geq  c $ for a critical value $ c $.

Polyhedral inference can be implemented analogously to the one-sided case based on the conditional distribution function
\footnotesize
\begin{align*}
   & F_{\text{2-sided}}(a,\theta, c ) =
 \Pr (X \leq a \mid |X| \geq  c ) 
 \\
 = & \frac{\min (\Phi (a - \theta), \Phi ( -  c  - \theta)) + \mathbf{1}_{a >  c } (\Phi(a- \theta) - \Phi( c  - \theta))}{\Phi(-  c  - \theta) + \Phi(-  c  + \theta)},
\end{align*}
\normalsize
where $\mathbf{1}_{a >  c }=1 $ if $a >  c $ and $=0$ otherwise.
Let $\theta_{\text{2-sided}}(p)$ denote the value of $\theta$ that solves $p = F_{\text{2-sided}}(x^{\text{obs}}, \theta,  c )$. 
The median-unbiased estimator is given by $\hat{\theta}_{MU, 2} = \theta_{\text{2-sided}}(0.5)$ and a $(1-\alpha)$-level confidence interval is given by $[\theta_{\text{2-sided}}(1-\alpha/2), \theta_{\text{2-sided}}(\alpha/2)]$.

In contrast to the one-sided case, $\theta_{\text{2-sided}}(p)$ converges to a finite limit as $x^{\text{obs}} \downarrow  c $. 
For example, at $p=0.5$, the limit solves the equation $0.5 = \Phi ( -  c  - \theta) / ( \Phi(-  c  - \theta) + \Phi(- c  + \theta))$ and the median-unbiased estimate is given by $\hat{\theta}_{MU, 2} = 0$ as $x^{\text{obs}} \downarrow  c $.

\begin{figure*}
    \centering
    \caption{Test statistic and confidence intervals conditional on two-sided significance}
    \label{fig: two-sided numerical illustration}
    \includegraphics[width = 0.8\linewidth]{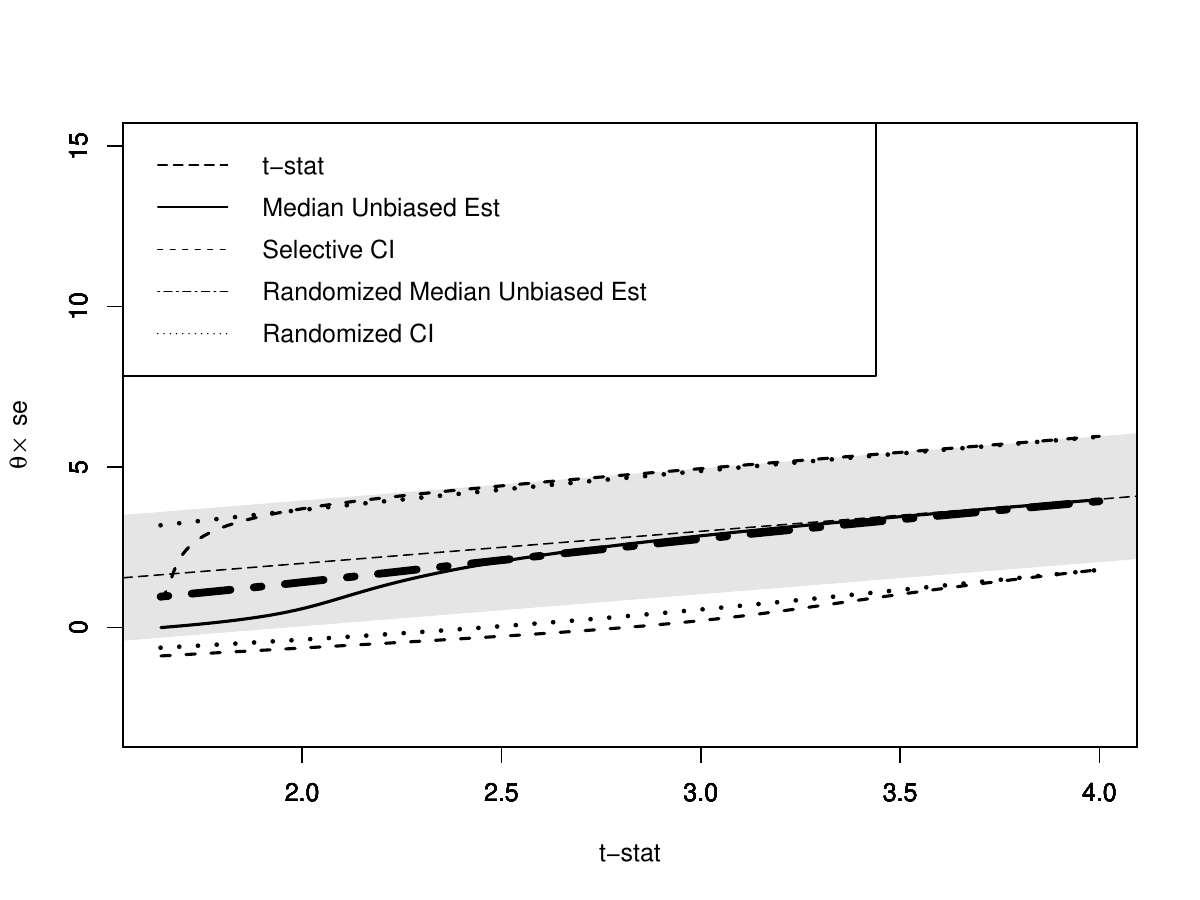}
    \begin{minipage}{0.8\linewidth}
    \footnotesize      NOTE: Conditional inference after two-sided testing with critical value $ c  = 1.64$ with and without randomized response ($\eta^2 = 1$). The dashed and dotted lines give the confidence intervals with and without randomized response, respectively. The solid and dash-dotted lines give the corresponding median-unbiased estimates. The thin dashed line is the 45-degree line, representing the unconditional estimate. The shaded region represents the conventional unconditional 95\% confidence interval.
    \end{minipage}
\end{figure*}

Figure \ref{fig: two-sided numerical illustration} displays the polyhedral conditional confidence interval and the corresponding median-unbiased estimate. For comparison, it also shows the confidence interval and point estimate based on a two-sided variant of the model with a randomized response discussed in Section~\ref{sec:randomized response}. The horizontal range of the figure are positive values of the $t$-statistic $x^{\text{obs}}$ for which we reject. A corresponding figure for negative values of $x^{\text{obs}}$ is obtained by reflecting the figure around the vertical axis.

Even without introducing randomization, the two-sided approach avoids the location problem observed in the one-sided case. For marginal significance, the conditional median-unbiased estimator and the conditional confidence interval no longer exhibit the extreme values that arise in the one-sided setting. In particular, the median-unbiased estimator converges to zero as we approach insignificance and the lower bound of the conditional confidence interval lies slightly below the lower bound of the conventional confidence interval. At most, it lies $0.847$ standard deviations below the conventional lower bound, which occurs at $x^{\text{obs}} \approx 2.8$. By contrast, in the one-sided case without randomization, both the median-unbiased estimator and the lower bound of the confidence interval diverge to negative infinity as the $t$-statistic approaches the critical value from above. 

Regardless of whether the selection contains a randomized response, the upper bound of the conditional confidence interval after two-sided selection lies below the conventional upper bound when $x^{\text{obs}}$ is small (for example, for values less than $2.0$). Approaching the critical value from above, the gap to the conventional upper bound becomes especially pronounced in the non-randomized case; eventually, it drops below the conventional estimator $\hat{\theta} = x^{\text{obs}}$, producing a confidence interval around zero whose length approaches only $1.77$ standard deviations. The conventional interval has length $2 \times 1.96 = 3.92$ standard deviations. Hence, conditional polyhedral inference can be substantially more precise than conventional inference. This stands in sharp contrast to the one-sided case, where polyhedral inference produces very wide confidence intervals at the margin of significance.

Comparing conditional intervals after selection with or without randomization, neither approach yields a uniformly tighter confidence interval. As discussed above, without randomization the conditional confidence interval contracts sharply at the margin of significance. By contrast, with randomization the interval is wider and more closely aligned with the unconditional interval. At moderate values of the $t$-statistic, however, the randomized interval is shorter than its non-randomized counterpart. The difference is largest at $0.44$ standard deviations, which occurs when $x^{\text{obs}} \approx 2.76$. From a practical standpoint, therefore, neither interval provides a decisive advantage, even in settings where both approaches are empirically viable.

\section{Conclusion}
\label{sec: conclusion}

In this paper, we examine the properties of selective inference methods. Our principal finding is that these methods can infer values or value ranges that are highly negative when the parameter of interest is only marginally significant. In many practical applications, such extreme values are ruled out by a-priori information about the parameter space. In this case, selective inference cannot infer any plausible parameter values.

While we consider a simple file drawer problem in this paper, results analogous to that presented in Section \ref{sec: theory} can be obtained for general polyhedral methods, provided that the conditioning set is bounded either from below or from above. This is implied by Lemma A.3 of \cite{kivaranovic2021}, which applies to general polyhedral methods. For example, conditional inference after variable selection by the Lasso is affected by this issue.

Our main result covers the case of selection by a one-sided test. We demonstrate that selection with a randomized response and selection based on a two-sided test guard against the location problem at the cost of redefining the selection mechanism.

\bibliographystyle{agsm} 
\bibliography{literature_stepdown.bib}       

@article{panigrahi21integrative,
	author = {Snigdha Panigrahi and Jonathan Taylor and Asaf Weinstein},
	doi = {https://doi.org/10.1214/21-AOS2057},
	journal = {The Annals of Statistics},
	number = {5},
	pages = {2803-2824},
	title = {Integrative Methods for Post-Selection Inference Using Convex Constraints},
	volume = {49},
	year = {2021}}

@article{brodeur2020methods,
	author = {Abel Brodeur and Nikolai Cook and Anthony Heyes},
	doi = {10.1257/aer.20190687},
	journal = {American Economic Review},
	number = {11},
	pages = {3634-3660},
	title = {Methods Matter: p-Hacking and Publication Bias in Causal Analysis in Economics},
	volume = {110},
	year = {2020}}

@article{rosenthal1979,
	author = {Robert Rosenthal},
	doi = {10.1037/0033-2909.86.3.638},
	journal = {Psychological Bulletin},
	number = {3},
	pages = {638-641},
	title = {The ``File Drawer Problem'' and Tolerance for Null Results},
	volume = {86},
	year = {1979}}

@article{kivaranovic2024,
	author = {Danijel Kivaranovic and Hannes Leeb},
	doi = {10.1214/24-EJS2232},
	journal = {Electronic Journal of Statistics},
	pages = {1677-1701},
	title = {A (tight) Upper Bound for the Length of Confidence Intervals with Conditional Coverage},
	volume = {18},
	year = {2024}}

@article{kivaranovic2021,
	author = {Kivaranovic ,Danijel and Leeb ,Hannes},
	doi = {10.1080/01621459.2020.1732989},
	isbn = {0162-1459},
	journal = {Journal of the American Statistical Association},
	journal1 = {Journal of the American Statistical Association},
	number = {534},
	pages = {845--857},
	title = {On the Length of Post-Model-Selection Confidence Intervals Conditional on Polyhedral Constraints},
	type = {doi: 10.1080/01621459.2020.1732989},
	url = {https://doi.org/10.1080/01621459.2020.1732989},
	volume = {116},
	year = {2021},
	year1 = {2021}}

@article{duy2023,
	author = {Duy, Vo Nguyen Le and Takeuchi, Ichiro},
	doi = {10.1007/s10463-022-00837-3},
	isbn = {1572-9052},
	journal = {Annals of the Institute of Statistical Mathematics},
	number = {1},
	pages = {127--157},
	title = {Exact statistical inference for the Wasserstein distance by selective inference},
	url = {https://doi.org/10.1007/s10463-022-00837-3},
	volume = {75},
	year = {2023}}

@article{terada2023,
	author = {Terada, Yoshikazu and Shimodaira, Hidetoshi},
	doi = {10.1007/s10463-022-00838-2},
	isbn = {1572-9052},
	journal = {Annals of the Institute of Statistical Mathematics},
	number = {1},
	pages = {99--125},
	title = {Selective Inference after Feature Selection via Multiscale Bootstrap},
	url = {https://doi.org/10.1007/s10463-022-00838-2},
	volume = {75},
	year = {2023}}

@article{watanabe2021,
	author = {Chihiro Watanabe and Taiji Suzuki},
	doi = {10.1214/21-EJS1853},
	journal = {Electronic Journal of Statistics},
	number = {1},
	pages = {3137 -- 3183},
	publisher = {Institute of Mathematical Statistics and Bernoulli Society},
	title = {Selective inference for latent block models},
	url = {https://doi.org/10.1214/21-EJS1853},
	volume = {15},
	year = {2021}}

@article{zhang2022post,
	author = {Zhang, Dongliang and Khalili, Abbas and Asgharian, Masoud},
	doi = {10.1214/22-SS135},
	journal = {Statistics Surveys},
	pages = {86--136},
	title = {Post-Model-Selection Inference in Linear Regression Models: An Integrated Review},
	volume = {16},
	year = {2022}}

@article{lockhart_significance_2014,
	author = {Lockhart, Richard and Taylor, Jonathan and Tibshirani, Ryan J. and Tibshirani, Robert},
	doi = {10.1214/13-AOS1175},
	issn = {0090-5364},
	journal = {Annals of Statistics},
	number = {2},
	pages = {413--468},
	title = {A {Significance} {Test} for the {Lasso}},
	url = {https://www.jstor.org/stable/43556287},
	urldate = {2024-08-05},
	volume = {42},
	year = {2014}}

@article{panigrahi_approximate_2023,
	author = {Panigrahi, Snigdha and Taylor, Jonathan},
	doi = {10.1080/01621459.2022.2081575},
	issn = {0162-1459},
	journal = {Journal of the American Statistical Association},
	month = oct,
	note = {Publisher: Taylor \& Francis \_eprint: https://doi.org/10.1080/01621459.2022.2081575},
	number = {544},
	pages = {2810--2820},
	title = {Approximate {Selective} {Inference} via {Maximum} {Likelihood}},
	url = {https://doi.org/10.1080/01621459.2022.2081575},
	urldate = {2024-07-08},
	volume = {118},
	year = {2023}}

@article{tibshirani_exact_2016,
	author = {Tibshirani, Ryan J. and Taylor, Jonathan and Lockhart, Richard and Tibshirani, Robert},
	doi = {10.1080/01621459.2015.1108848},
	issn = {0162-1459},
	journal = {Journal of the American Statistical Association},
	number = {514},
	pages = {600--620},
	title = {Exact {Post}-{Selection} {Inference} for {Sequential} {Regression} {Procedures}},
	url = {https://doi.org/10.1080/01621459.2015.1108848},
	urldate = {2024-07-08},
	volume = {111},
	year = {2016}}

@article{andrews_inference_2024,
	author = {Andrews, Isaiah and Kitagawa, Toru and McCloskey, Adam},
	doi = {10.1093/qje/qjad043},
	issn = {0033-5533},
	journal = {The Quarterly Journal of Economics},
	month = feb,
	number = {1},
	pages = {305--358},
	title = {Inference on {Winners}*},
	url = {https://doi.org/10.1093/qje/qjad043},
	urldate = {2024-04-17},
	volume = {139},
	year = {2024}}

@misc{fithian_optimal_2017,
	author = {Fithian, William and Sun, Dennis and Taylor, Jonathan},
	doi = {10.48550/arXiv.1410.2597},
	note = {arXiv:1410.2597 [math, stat]},
	title = {Optimal {Inference} {After} {Model} {Selection}},
	url = {http://arxiv.org/abs/1410.2597},
	urldate = {2024-04-17},
	year = {2017}}

@article{tian_selective_2018,
	author = {Tian, Xiaoying and Taylor, Jonathan},
	doi = {10.1214/17-AOS1564},
	issn = {0090-5364},
	journal = {Annals of Statistics},
	number = {2},
	pages = {679--710},
	title = {Selective {Inference} with a {Randomized} {Response}},
	url = {https://www.jstor.org/stable/26542802},
	urldate = {2024-08-05},
	volume = {46},
	year = {2018}}

@article{jewell_testing_2022,
	author = {Jewell, Sean and Fearnhead, Paul and Witten, Daniela},
	doi = {10.1111/rssb.12501},
	issn = {1369-7412},
	journal = {Journal of the Royal Statistical Society Series B: Statistical Methodology},
	number = {4},
	pages = {1082--1104},
	title = {Testing for a {Change} in {Mean} after {Changepoint} {Detection}},
	url = {https://doi.org/10.1111/rssb.12501},
	urldate = {2024-08-05},
	volume = {84},
	year = {2022}}

@article{bachoc_valid_2019,
	author = {Bachoc, Fran{\c c}ois and Leeb, Hannes and P{\"o}tscher, Benedikt M.},
	doi = {10.1214/18-AOS1721},
	issn = {0090-5364},
	journal = {Annals of Statistics},
	number = {3},
	pages = {1475--1504},
	title = {Valid {Confidence} {Intervals} for {Post}-{Model}-{Selection} {Predictors}},
	url = {https://www.jstor.org/stable/26730430},
	urldate = {2024-08-05},
	volume = {47},
	year = {2019}}

@article{berk_valid_2013,
	author = {Berk, Richard and Brown, Lawrence and Buja, Andreas and Zhang, Kai and Zhao, Linda},
	doi = {10.1214/12-AOS1077},
	issn = {0090-5364},
	journal = {Annals of Statistics},
	number = {2},
	pages = {802--837},
	title = {Valid {Post}-{Selection} {Inference}},
	url = {https://www.jstor.org/stable/23566582},
	urldate = {2024-08-05},
	volume = {41},
	year = {2013}}

@article{lee2016exact,
	author = {Lee, Jason D. and Sun, Dennis L. and Sun, Yuekai and Taylor, Jonathan E.},
	doi = {10.1214/15-AOS1371},
	journal = {Annals of Statistics},
	number = {3},
	pages = {907--927},
	publisher = {Institute of Mathematical Statistics},
	title = {Exact post-selection inference, with application to the lasso},
	volume = {44},
	year = {2016}}

@article{brodeur2016star,
	author = {Brodeur, Abel and L{\'e}, Mathias and Sangnier, Marc and Zylberberg, Yanos},
	doi = {10.1257/app.20150044},
	journal = {American Economic Journal: Applied Economics},
	number = {1},
	pages = {1--32},
	publisher = {American Economic Association 2014 Broadway, Suite 305, Nashville, TN 37203-2425},
	title = {Star wars: The empirics strike back},
	volume = {8},
	year = {2016}}

@article{gerber2008statistical,
	author = {Gerber, Alan and Malhotra, Neil},
	doi = {10.1561/100.00008024},
	journal = {Quarterly Journal of Political Science},
	number = {3},
	pages = {313--326},
	publisher = {Citeseer},
	title = {Do statistical reporting standards affect what is published? Publication bias in two leading political science journals},
	volume = {3},
	year = {2008}}


\end{document}